\documentclass[10pt,a4paper]{article}

\usepackage{amsmath}
\usepackage{amsfonts}
\usepackage{amssymb}
\usepackage{amsthm}
\usepackage{hyperref}
\usepackage{a4wide}
\usepackage{enumerate}
\usepackage[usenames,dvipsnames]{xcolor}
\usepackage{pgf,tikz,tkz-graph}
\usetikzlibrary{arrows,shapes}
\usetikzlibrary{decorations.pathreplacing}
\usepackage{tkz-berge}

\newtheorem{thm}{Theorem}[section]
\newtheorem{cor}[thm]{Corollary}
\newtheorem{lem}[thm]{Lemma}
\newtheorem{prop}[thm]{Proposition}
\newtheorem{obs}[thm]{Observation}

\newtheorem{conj}[thm]{Conjecture}
\newtheorem*{conj1redux}{Conjecture~\ref{conj:AlKr} redux}
\newtheorem*{claim}{Claim}
\newtheorem*{thm*}{Theorem}
\newtheorem{prob}[thm]{Problem}

\theoremstyle{definition}

\newcommand*{\myproofname}{Proof}
\newenvironment{claimproof}[1][\myproofname]{\begin{proof}[#1]}{\end{proof}}

\newcommand*{\ceilfrac}[2]{\mathopen{}\left\lceil\frac{#1}{#2}\right\rceil\mathclose{}}

\newcommand*{\bceil}[1]{\left\lceil #1\right\rceil}

\begin{document}

\title{A precise condition for independent transversals in bipartite covers\footnote{A preliminary version of this work, absent of the main proofs, appeared in the Proceedings of the 12th European Conference on Combinatorics, Graph Theory and Applications ({\em EuroComb 2023}, Prague).}}

\author{Stijn Cambie\thanks{
Department of Computer Science, KU Leuven Campus Kulak Kortrijk, Belgium. Partially supported by 
the Institute for Basic Science (IBS-R029-C4) and by Internal Funds of KU Leuven (PDM fellowship PDMT1/22/005).
\protect\href{mailto:stijn.cambie@hotmail.com}{\protect\nolinkurl{stijn.cambie@hotmail.com}}}
\and
Penny Haxell\thanks{Department of Combinatorics and Optimization, University of Waterloo, Waterloo, ON Canada N2L 3G1. Partially supported by NSERC. \protect\href{mailto:pehaxell@uwaterloo.ca}{\protect\nolinkurl{pehaxell@uwaterloo.ca}}}
\and
Ross J. Kang\thanks{Korteweg--de Vries Institute for Mathematics, University of Amsterdam, PO Box 94248, 1090 GE Amsterdam, Netherlands. Partially supported by a Vidi grant (639.032.614) and the Gravitation Programme NETWORKS (024.002.003) of the Dutch Research Council (NWO). \protect\href{mailto:r.kang@uva.nl}{\protect\nolinkurl{r.kang@uva.nl}}}
\and
Ronen Wdowinski\thanks{Department of Combinatorics and Optimization, University of Waterloo, Waterloo, ON Canada N2L 3G1. \protect\href{mailto:ronen.wdowinski@uwaterloo.ca}{\protect\nolinkurl{ronen.wdowinski@uwaterloo.ca}}}}

\date{\today}

\maketitle

\begin{abstract}
Given a bipartite graph $H=(V=V_A\cup V_B,E)$ in which any vertex in $V_A$ (resp.~$V_B$) has degree at most $D_A$ (resp.~$D_B$), suppose there is a partition of $V$ that is a refinement of the bipartition $V_A\cup V_B$ such that the parts in $V_A$ (resp.~$V_B$) have size at least $k_A$ (resp.~$k_B$). We prove that the condition $D_A/k_B+D_B/k_A\le 1$ is sufficient for the existence of an independent set of vertices of $H$ that is simultaneously transversal to the partition, and show moreover that this condition is sharp.
This result is a bipartite refinement of two well-known results on independent transversals, one due to the second author and the other due to Szab\'o and Tardos.
\end{abstract}


\section{Introduction}

Consider the following question: how much easier is it to colour graphs that are bipartite than to colour graphs in general? Of course, when considered in the context of the usual chromatic number, this is utterly trivial: compared to the general case, for which the chromatic number can be $\Delta(G)+1$ but no larger (with $\Delta(G)$ denoting the maximum degree of $G$), the factor of reduction in the number of necessary colours is of order $\Delta(G)$. We treat some settings stronger than that of ordinary proper colouring, settings that have both classic and contemporary combinatorial motivation.

Recall the definition of the list chromatic number, a notion introduced nearly half a century ago independently by Erd\H{o}s, Rubin and Taylor~\cite{ERT80} and by Vizing~\cite{Viz76}. Let $G = (V,E)$ be a simple, undirected graph. 
A mapping $L: V(G)\to 2^{{\mathbb Z}^+}$ is called a {\em list-assignment} of $G$. If for some positive integer $k$, the mapping $L$ satisfies $|L(v)|=k$ for all $v$, then it is called a {\em $k$-list-assignment}.
A colouring $c: V\to {\mathbb Z}^+$ is called an {\em $L$-colouring} if $c(v)\in L(v)$ for any $v\in V$. We say $G$ is {\em $k$-choosable} if for any $k$-list-assignment $L$ of $G$ there is a proper $L$-colouring of $G$. The {\em choosability $\chi_\ell(G)$} (or {\em choice number} or {\em list chromatic number}) of $G$ is the least $k$ such that $G$ is $k$-choosable.

Framing the above question with respect to the list chromatic number, note first that a greedy procedure implies $\chi_\ell(G) \le \Delta(G)+1$ always,  which is exact for $G$ a complete graph. However, for bipartite $G$, it is a longstanding conjecture that $\chi_\ell(G)$ must be lower than this bound by a factor of order $\Delta(G)/\log \Delta(G)$.

\begin{conj}[Alon and Krivelevich~\cite{AlKr98}]\label{conj:AlKr}
There is some $C\ge 1$ such that $\chi_\ell(G) \le C\log_2 \Delta(G)$ for any bipartite graph $G$ with $\Delta(G)\ge2$.
\end{conj}

\noindent
If true, this statement would be sharp up to the value of $C$, due to the complete bipartite graphs~\cite{ERT80}.
For an idea of how stubborn this problem has been, we relate to the reader how the current best progress was essentially {\em already known} to the conjecture's originators.
In particular, a seminal result for triangle-free graphs of Johansson~\cite{Joh96} from the mid-1990's implies that $\chi_\ell(G) = O(\Delta(G)/\log\Delta(G))$ as $\Delta(G)\to\infty$ for any bipartite $G$, hence a reduction factor only of order $\log\Delta(G)$. 

To stimulate activity, two of the authors with Alon~\cite{ACK21,CaKa22} proposed some natural refinements and variations of Conjecture~\ref{conj:AlKr}, and offered modest related progress. Although less directly relevant to Conjecture~\ref{conj:AlKr}, the present work has the momentum of this trajectory. We introduce some definitions needed to properly describe this progression. In particular, we cast the (bipartite) colouring task in a more precise and general way.

Let $G$ and $H$ be simple, undirected graphs.
We say that $H$ is a {\em cover (graph)} of $G$ with respect to a mapping $L: V(G) \to 2^{V(H)}$ if $L$ induces a partition of $V(H)$ and the bipartite subgraph induced between $L(v)$ and $L(v')$ is edgeless whenever $vv'\notin E(G)$.
If for some positive integer $k$, the mapping $L$ satisfies $|L(v)|=k$ for all $v$, then we call $H$ a {\em $k$-fold} cover of $G$ (with respect to $L$).
Moreover, if $G$ and $H$ are bipartite graphs, where $G$ admits a bipartition $V(G)=A_G\cup B_G$ and $H$ admits a bipartition $V(H)=A_H\cup B_H$, then we say that $H$ is a {\em bipartite-cover (graph)} of $G$ with respect to $L$ if $L(A_G)$ induces a partition of $A_H$ and $L(B_G)$ induces a partition of $B_H$, i.e.~the bipartitions of $G$ and $H$ suitably align.
We will have reason to be even more specific for this situation by referring to $H$ as an {\em $(A,B)$-cover} of $G$ (with respect to $L$). 
(Here we regard $A$ as the pair $(A_G,A_H)$ of partitions, and $B$ similarly.)
To denote this setup succinctly we write
$H=H(A,B,G,L)$ for the bipartite-cover. 

To connect the notions above to Conjecture~\ref{conj:AlKr}, notice that, for any list-assignment $L$ of some graph $G$, one may construct a cover graph $H$ as follows. The vertices of $H$ consist of all pairs $(v,x)$ for $v\in G$ and $x\in L(v)$, and $E(H)$ is a subset of the collection of pairs $(v,x)(v',x')$ such that $vv'\in E(G)$ and $x=x'$. By regarding $L$ as a mapping from $v$ to $\{(v,x) \mid x\in L(v)\}$, we can then regard $H$ as a cover graph of $G$ with respect to $L$. 
Moreover, if $G$ is bipartite, the corresponding $H$ is a bipartite-cover of $G$ with respect to $L$.
We refer to any (bipartite) cover graph constructed as above as a {\em (bipartite) list-cover}.
Moreover, if $E(H)$ is chosen maximally, we may refer to $H$ as the {\em maximal (bipartite) list-cover} of $G$ with respect to $L$.
Notice that a proper $L$-colouring of $G$ is equivalent to an independent set in the corresponding maximal list-cover $H$ that is transversal to the partition induced by $L$ (that is, it intersects every part exactly once). We call such an independent set an {\em independent transversal (IT)} of $H$.

\begin{conj1redux}
There is some $C\ge 1$ such that, for any bipartite graph $G$ of 
maximum degree $\Delta\ge2$, any $\bceil{ C\log_2 \Delta } $-fold
bipartite list-cover of $G$ admits an independent transversal.
\end{conj1redux}

There are three potential directions to  highlight through adoption of the above terminology.
First, note that list-covers form a proper subclass of all cover graphs,
and so we might consider the `colouring' task under increasingly more general conditions with respect to $H$. More specifically, we may ask analogous questions about sufficient conditions for the existence of an IT in natural and successively larger superclasses of list-covers (among all cover graphs).
Second, note that if $G$ has maximum degree at most $\Delta$, then a list-cover of $G$ has maximum degree at most $\Delta$, 
but the converse is not true in general.
Hence, for instance, we may consider a problem/result about list-colouring in some class of bounded degree graphs and try to generalise it to the analogous class of bounded degree list-covers. This type of `colour-degree' problem was introduced by Reed~\cite{Ree99}.
Third, and specific to $(A,B)$-covers, we may insist on a more refined viewpoint by imposing (degree/list-size/structural) conditions on parts $A$ and $B$ separately.
Two of the authors together with Alon~\cite{ACK21} introduced this third asymmetric perspective for studying Conjecture~\ref{conj:AlKr}, and in a follow up~\cite{CaKa22} they furthermore took on the first two perspectives, in particular by generalising the problem to so-called {\em correspondence-covers}, which we discuss later.
Here we concentrate on the most general case for (asymmetric, bipartite) cover graphs with given degree bounds.

The following problem was posed in~\cite{CaKa22} (see therein the special case of Problem~1.1 with $\Delta_A, \Delta_B$ infinite).
\begin{prob}\label{prob:CaKa22}
Let $H=H(A,B,G,L)$ be a bipartite-cover.
Determine conditions on positive integers $k_A$, $k_B$, $D_A$, $D_B$ that
suffice to ensure the following: If the maximum degrees in $A_H$ and
$B_H$ are $D_A$ and $D_B$, respectively, and $|L(v)| \ge k_A$ for all
$v \in A_G$ and $|L(w)| \ge k_B$ for all $w \in B_G$, then there is
guaranteed to be an independent transversal of $H$ with respect to
$L$. 
\end{prob}
\noindent
We resolve Problem~\ref{prob:CaKa22} through the following sufficient
condition for a bipartite-cover graph to admit an IT. 
\begin{thm}\label{thm:main}
Let $H=H(A,B,G,L)$ be a bipartite-cover.
Let positive integers $k_A$, $k_B$, $D_A$, $D_B$ be such that
$\frac{D_B}{k_A} + \frac{D_A}{k_B} \le 1$.
If the maximum degrees in $A_H$ and
$B_H$ are $D_A$ and $D_B$, respectively, and $|L(v)| \ge k_A$ for all
$v \in A_G$ and $|L(w)| \ge k_B$ for all $w \in B_G$, then $H$ admits
an independent transversal with respect to $L$. 
\end{thm}

\noindent
We show in Section~\ref{sec:sufficient} that this result is a corollary
to a general result for independent transversals found e.g.
in~\cite{Hax11}. 

In fact, the condition in Theorem~\ref{thm:main} is best possible, as follows.

\begin{thm}\label{thm:sharp}
Let positive integers $k_A$, $k_B$, $D_A$, $D_B$ be such that
$\frac{D_B}{k_A} + \frac{D_A}{k_B} > 1$.
Then there exists a bipartite-cover $H=H(A,B,G,L)$ such that the
maximum degrees in $A_H$ and 
$B_H$ are at most $D_A$ and $D_B$, respectively, and $|L(v)| = k_A$ for all
$v \in A_G$ and $|L(w)| = k_B$ for all $w \in B_G$, and such that $H$ admits no independent transversal with respect to $L$.
\end{thm}

\noindent
It is worth isolating the symmetric situation where we maintain that $D_A=D_B=D$ and $k_A=k_B=k$; in this case the condition in Theorem~\ref{thm:main} resolves to $k \ge 2D$. In other words, we have the following.
\begin{cor}\label{cor:symmetric}
Any $(2D)$-fold bipartite-cover graph of maximum degree $D$ admits an independent transversal. Moreover, the conclusion may fail if the $2D$ part size condition is relaxed to $2D-1$.
\end{cor}
\noindent
This condition coincides with that of a well-known, more general
result of the second author~\cite{Hax95, Hax01}: any $(2D)$-fold cover
graph of maximum degree $D$ is guaranteed to admit an IT. As such, one
may see Theorem~\ref{thm:sharp} as simultaneously a strengthening and
generalisation of a result of Szab\'o and Tardos~\cite{SzTa06} (which
in turn built upon a series of results beginning in the original paper
of Bollob\'as, Erd\H{o}s and Szemer\'edi~\cite{BES75}): there
exists a $(2D-1)$-fold cover graph of maximum degree $D$ that does not
admit an IT.

Recalling the question posed at the beginning, Theorem~\ref{thm:sharp}
and Corollary~\ref{cor:symmetric} show in a wider sense how the bipartite assumption does {\em not} help for the existence of ITs in cover graphs.
Indeed, while the construction of Szab\'o and Tardos is composed of the union of complete bipartite graphs, its partition classes do not align with a bipartition. Corollary~\ref{cor:symmetric} affirms that it is possible to achieve such an alignment in some bipartite construction. Our construction also consists of a union of complete bipartite graphs.
For an indication of the difference, Figure~\ref{fig:constr5:3} depicts the $D=3$ construction in Corollary~\ref{cor:symmetric}, and one can compare it with~\cite[Fig.~1]{SzTa06}.

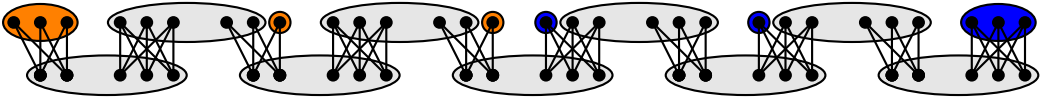
\begin{figure}
\centering
\begin{tikzpicture}[thick, scale=.7]

\foreach \x in {1,3}{
\begin{scope}[yscale=.25] 
\draw[fill=black!10!white] (3.75+2*\x,4) circle (1.48);
\end{scope}
}

\foreach \x in {5,7}{
\begin{scope}[yscale=.25] 
\draw[fill=black!10!white] (4.25+2*\x,4) circle (1.48);
\end{scope}
}

\begin{scope}[yscale=0.5] 
\draw[fill=orange] (3,2) circle (0.7);
\end{scope}

\draw[fill=orange] (7.5,1) circle (0.2);
\draw[fill=orange] (11.5,1) circle (0.2);
\draw[fill=blue] (16.5,1) circle (0.2);
\draw[fill=blue] (12.5,1) circle (0.2);

\begin{scope}[yscale=0.5] 
\draw[fill=blue] (21,2) circle (0.7);
\end{scope}

\foreach \x in {1,3,...,9}{
\begin{scope}[yscale=.25] 
\draw[fill=black!10!white] (2.25+2*\x,0) circle (1.5);
\end{scope}
}

\foreach \x in {1,3,...,9}{
\foreach \y in {1,2,3}{
\foreach \z in {2,3}{
\draw[thick] (2*\x+0.5*\z,0) -- (2*\x+0.5*\y,1);
\draw[fill] (2*\x+0.5*\z,0) circle (0.1);	
}
\draw[fill] (2*\x+0.5*\y,1) circle (0.1);	
}
}

\foreach \x in {2,4,...,10}{
\foreach \y in {1,2,3}{
\foreach \z in {1,2,3}{
\draw[thick] (2*\x+0.5*\z,0) -- (2*\x+0.5*\y,1);
}
\draw[fill] (2*\x+0.5*\y,1) circle (0.1);	
\draw[fill] (2*\x+0.5*\y,0) circle (0.1);	
}
}

\end{tikzpicture}
\caption{A bipartite-cover graph with maximum degree $3$ and partition classes of size $5$ with no IT. The partition classes are represented by ovals, except that the vertices contained in orange are all a single partition class, as are the vertices contained in blue.}\label{fig:constr5:3}
\end{figure}

Let us briefly discuss what happens in the special case of correspondence-covers, as explored in~\cite{CaKa22}. Given a cover graph $H$ of $G$ with respect to $L$, we say $H$ is a {\em correspondence-cover} if the bipartite subgraph induced between $L(v)$ and $L(v')$ is a matching for any $vv'\in V(G)$. In other words, the maximum degree induced between two parts of $H$ with respect to $L$ is at most $1$. Clearly the class of all correspondence-covers strictly includes that of all list-covers. The next result follows from a `coupon collector' argument, and this is counterbalanced by a simple probabilistic construction (that was given, for example, in~\cite{KPV05}).

\begin{thm}[\cite{CaKa22}]\label{thm:coupon}
For any $\varepsilon>0$, the following holds for all $D$ sufficiently large.
Any $\bceil{(1+\varepsilon)\frac{D}{\log D}}$-fold bipartite correspondence-cover graph of maximum degree $D$ admits an independent transversal. Moreover, the conclusion fails if the $(1+\varepsilon)$ factor is weakened to a $(\frac12-\varepsilon)$ factor.
\end{thm}

\noindent
One reason for highlighting this case is that it could be interesting
to gradually tune (between $1$ and $D$) the condition on maximum
degree induced between two parts of $H$ with respect to $L$, in order
to gain a better understanding of the transition between the
$\Theta(D/\log D)$ (probabilistic) part-size condition in
Theorem~\ref{thm:coupon} and the $\Theta(D)$ condition in
Corollary~\ref{cor:symmetric} (which was originally established
in~\cite{Hax95, Hax01}). 

In another setting `between' the correspondence-cover and general
cover problems, the following result~\cite[Prop~1.7]{CaKa22} was found
to be a direct corollary of the result of~\cite{Hax95, Hax01}. 

\begin{prop}[\cite{CaKa22}]\label{prop:CaKa22}
Let $H=H(A,B,G,L)$ be a bipartite-cover.
Let positive integers $k_A$, $k_B$, and $D$ be such that
$k_A=2D^2$ and $k_B=2$. If $H$ has maximum degree $D$, and
$|L(v)| \ge k_A$ for all
$v \in A_G$ and $|L(w)| = k_B$ for all $w \in B_G$,
and no vertex of $A_H$ is adjacent to both vertices of $L(w)$ for some
$w\in B_G$, then $H$ admits an independent transversal with respect to $L$.
\end{prop}

\noindent
In Section~\ref{df}, we give a construction to show how the value
$2D^2$ cannot be lowered to $2D^2-1$ for the same conclusion.

Let us conclude by returning to the original motivation and a related challenge.
With Corollary~\ref{cor:symmetric} and Theorem~\ref{thm:coupon} in mind, the following `colour-degree' generalisation of Conjecture~\ref{conj:AlKr} seems worth investigating.

\begin{conj}\label{conj:AlKrcolourdegree}
There exists some $C\ge 1$ such that for every $D\geq 2$, every $\bceil{C\log_2
D}$-fold bipartite list-cover graph of maximum degree $D$ admits an independent
transversal. 
\end{conj}

\noindent
To round out the story, we point out how Conjectures~\ref{conj:AlKr} and~\ref{conj:AlKrcolourdegree} are essentially equivalent.

\begin{thm}
If Conjecture~\ref{conj:AlKr} is true for some constant $C \ge 1$, then Conjecture~\ref{conj:AlKrcolourdegree} is true for some constant $C'\ge 1.$
The same implication holds when $C$ and $C'$ are both replaced by $1+o(1)$ (as $\Delta,D\to\infty$).
\end{thm}

\begin{proof}
Assume Conjecture~\ref{conj:AlKr} is true for some $C \ge 1$.
We choose $D_0\ge 2$ such that $\sqrt D \ge C \log_2 D$ for every $D
\ge D_0$, and set $C'=2D_0 \ge 2C^2\ge2C$.

Let $D\geq 2$ be given. Set $k=C' \log_2 D$ and
let $H=H(A,B,G,L)$ be an
arbitrary $k$-fold bipartite list-cover of maximum degree $D \ge
2$. We will show that $H$ has an IT.
If $D \le D_0,$ then $k \ge C' \ge 2D$, and
Corollary~\ref{cor:symmetric} implies that $H$ has an independent
transversal as desired. We may therefore assume $D> D_0.$
By definition, the maximum degree $\Delta$ of the minimal covered graph $G$
satisfies $D_0 \le D \le \Delta \le kD.$ 
By the choice of $D_0$ it then follows that $\sqrt \Delta \ge C \log_2
\Delta$.

We claim that $k\geq  C \log_2 \Delta$. To see this, assume the contrary,
and note then that $\Delta < D\cdot C\log_2\Delta$, and hence $D >
\frac{\Delta}{C \log_2 \Delta} \ge \sqrt{\Delta}$. 
But $D> \sqrt{\Delta}$ and $C'\ge 2C$ imply that $k=C' \log_2 D \ge
C\log_2 \Delta$, which is a contradiction. 
Therefore $k\ge  C \log_2 \Delta$ as claimed.

Now consider the maximal list-cover $H'\supseteq H$ of $G$ with respect to $L$.
Our claim that $k\geq  C \log_2 \Delta$, together with the assumption that
Conjecture~\ref{conj:AlKr} is true with the 
constant $C$, imply that $H'$ admits an independent transversal with
respect to $L$. Hence the same is true for $H$, as required.

The proof for the $1+o(1)$ version proceeds analogously.
Fix $\varepsilon>0$.
Now one can take $D_0$ sufficiently large such that 
$(1+\varepsilon) \log_2 D \le D^{\varepsilon}$ for all $D\geq D_0$, and such that Conjecture~\ref{conj:AlKr} is true with $1+\varepsilon$ whenever $\Delta \ge D_0.$
Then for any $k$-fold bipartite list-cover $H$ of maximum degree $D \ge 2$, where $k\ge  \frac{1+\varepsilon}{1-\varepsilon} \log_2 D$, we conclude $H$ has an independent transversal by the same strategy.
\end{proof}

\noindent
In a similar way, non-trivial progress on Conjecture~\ref{conj:AlKr} may imply non-trivial progress on Conjecture~\ref{conj:AlKrcolourdegree}.
Conversely, lower bound constructions related to Conjecture~\ref{conj:AlKrcolourdegree} may directly yield corresponding constructions related to Conjecture~\ref{conj:AlKr}.

\section{A sufficient condition}\label{sec:sufficient}

In this section, we derive Theorem~\ref{thm:main}.

We say that a set $U$ of vertices of a graph $G$ {\it dominates} the
set $W\subseteq V(G)$ if every vertex of
$W$ has a neighbour in 
$U$. (This is a somewhat nonstandard use of the term since, contrary to the
most common usage, here we
require each vertex of $U\cap W$ to have a neighbour in $U$.) 
Theorem~\ref{thm:main} is a straightforward consequence of
the following result, which is implicit in the proof of~\cite{Hax95,
  Hax01} and stated explicitly e.g.~\cite{AhBeZi} (see Theorem 3.5 therein, which is written in slightly different language). See also e.g.~\cite{BeHaSz, Hax11} for a stronger statement.

\begin{thm} \label{Minimal-IT-lemma}
Let $H=(V_H,E_H)$ be a cover graph of some graph $G=(V_G,E_G)$ with
respect to $L$, that has no IT.
Then there exists a subset $S\subseteq V_G$ and a set $Z$ of edges of the subgraph of $H$ induced by $L(S)$ such that
$|Z|\le |S|-1$, and $V_H(Z)$ dominates $L(S)$.
\end{thm}

\begin{proof}[Proof of Theorem~\ref{thm:main}]
Suppose $H$ is a counterexample. By
Theorem~\ref{Minimal-IT-lemma}, there exist some $a$ partition classes
of $A_H$ and
$b$ partition classes of $B_H$, and a set $Z$ of edges of size at most
$a + b - 1$ whose end-vertices dominate the union of these $a+b$ partition
classes. The end-vertices of $Z$ dominate at most $(a + b -
1)D_B$ vertices in $A_H$, while the $a$ partition classes contain at least
$ak_A$ vertices. This implies that $ak_A \le (a + b - 1)D_B$. Similarly, considering $B_H$,
we have $bk_B \le (a + b - 1)D_A$. But then 
\begin{align*}
	\frac{D_A}{k_B} + \frac{D_B}{k_A} \ge \frac{b}{a + b - 1} +
        \frac{a}{a + b - 1} > 1, 
\end{align*}
contradicting the hypothesis.
\end{proof}

\section{Sharpness of the condition}

In this section, we have chosen, for ease of notation, to drop reference to the `covered' graph and the function $L$, and instead refer mainly to graphs with vertex partitions.

To prove Theorem~\ref{thm:sharp}, we will make frequent use of
the following lemma, which was also used in \cite{HaxWd2} and explored in greater detail in \cite{HaxWd1}.

\begin{lem} \label{join-bipartite}
Let $H$ and $J$ be disjoint vertex-partitioned graphs with no IT and
$m(H), m(J) \ge 1$ partition classes, respectively. Construct a vertex
partition of the graph $G = H \cup J$ with $m(H) + m(J) - 1$ classes
by choosing one partition class of $J$ and arbitrarily distributing
its vertices into the partition classes of $H$. Then $G$ has no
IT. 
\end{lem}

\begin{proof}
Suppose for contradiction that $G$ has an IT $X \subseteq V(G)$. Since
$|X| = m(H) + m(J) - 1$, we know that $X$ contains either (at least) $m(H)$
vertices of $H$ or 
$m(J)$ vertices of $J$. If $X$ contains $m(H)$ vertices of $H$, then
since the vertex partition of $H$ is unchanged in $G$, it follows that
$X \cap V(H)$
is an IT of $H$. This contradiction shows that $X$ contains $m(J)$ vertices of
$J$. Then at least one of these vertices is contained in the class of $J$
that was distributed among the classes of $H$, because all other
partition classes of $J$ are unchanged and hence must each intersect
$X$ exactly once. But then $X\cap V(J)$ contains an IT of $J$, again a
contradiction. 
%
\end{proof}

To simplify the description of the upcoming construction, we introduce
some terminology. Let $k_A, k_B, D_A, D_B$ be positive
integers. A $(k_A,k_B,D_A,D_B)$-{\it graph} is a bipartite graph $H$ with
parts $A$ and $B$, together with a partition of $A$ and of $B$, such that
each vertex in $A$ (respectively $B$) has degree at most $D_A$
(respectively $D_B$), and each $A$-class (respectively $B$-class) has
size at most $k_A$ (respectively $k_B$). An $A$-class (respectively
$B$-class) is said to be {\it full} if it has size $k_A$ (respectively
$k_B$), and $H$ is {\it full} if all its classes are full. Thus our
aim is to construct a full $(k_A, k_B, D_A, 
D_B)$-graph $H$ with no IT, under the condition 
\begin{align} \label{ratio-condition}
	\frac{D_B}{k_A} + \frac{D_A}{k_B} > 1. \tag{$\ast$}
\end{align}
An example of a full $(5,6,3,3)$-graph with no IT is presented in Figure~\ref{fig:constr_(5,6,3,3}.

\begin{figure}[!h]
\centering
\begin{tikzpicture}[thick, scale=.7]

\foreach \x in {1,3}{
\begin{scope}[yscale=.25] 
\draw[fill=black!10!white] (3.75+2*\x,4) circle (1.48);
\end{scope}
}

\foreach \x in {5,7}{
\begin{scope}[yscale=.25] 
\draw[fill=black!10!white] (4.25+2*\x,4) circle (1.48);
\end{scope}
}

\begin{scope}[yscale=0.5] 
\draw[fill=orange] (3,2) circle (0.7);
\end{scope}

\draw[fill=orange] (7.5,1) circle (0.2);
\draw[fill=orange] (11.5,1) circle (0.2);
\draw[fill=blue] (16.5,1) circle (0.2);
\draw[fill=blue] (12.5,1) circle (0.2);

\begin{scope}[yscale=0.5] 
\draw[fill=blue] (21,2) circle (0.7);
\end{scope}

\foreach \x in {1,3,...,9}{
\begin{scope}[yscale=.25] 
\draw[fill=black!10!white] (2+2*\x,0) circle (1.75);
\end{scope}
}

\foreach \x in {1,2,...,10}{
\foreach \y in {1,2,3}{
\foreach \z in {1,2,3}{
\draw[thick] (2*\x+0.5*\z,0) -- (2*\x+0.5*\y,1);
}
\draw[fill] (2*\x+0.5*\y,1) circle (0.1);	
\draw[fill] (2*\x+0.5*\y,0) circle (0.1);	
}
}

\end{tikzpicture}
\caption{A full $(5,6,3,3)$-graph, where the partition classes are
  represented by ovals, except that the vertices contained in orange
  are all a single partition class, as are the vertices contained in
  blue. Note that the graph in Figure 1 is a subgraph of this example.}\label{fig:constr_(5,6,3,3}
\end{figure}
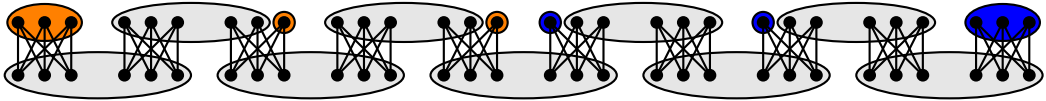

First we specify some simple assumptions that we may make in
order to construct the desired graph, given the parameters $k_A, k_B,
D_A, D_B$.  
\begin{obs}\label{easyob} 
  Let positive integers $k_A, k_B, D_A, D_B$ be given, such that
Condition~(\ref{ratio-condition}) holds. Then without loss of
generality we may assume each of the following.
\begin{enumerate}
\item There exists a positive integer $t$ such that $k_AD_A + k_BD_B = k_Ak_B+t$,
\item $D_A \leq k_B \leq2D_A-1$,
\item $k_A\geq 2D_B$.
\end{enumerate}
\end{obs}

\begin{proof}
  Part (1) is immediate from Condition~(\ref{ratio-condition}) and the
  fact that all parameters are integers. For the second inequality of 
  Part (2), if on the contrary both $k_A\geq 2D_B$ and $k_B\geq 2D_A$ 
  hold, then Condition~(\ref{ratio-condition}) is violated. Thus we may
  assume $k_B \leq2D_A-1$. For the first inequality of 
  Part (2), if $D_A > k_B$, then Condition~(\ref{ratio-condition}) still holds if we replace $D_A$ by $D_A' = k_B$. Thus if we could construct
  a full $(k_A, k_B, D_A', D_B)$-graph $H$ with no IT, then $H$ would
  also be a full $(k_A, k_B, D_A, D_B)$-graph with no IT. Thus we may
  assume $D_A \le k_B$. For Part (3),
  suppose on the contrary that $k_A \leq2D_B-1$. Then using (2),
\begin{align*}
\frac{D_B}{k_A+1} + \frac{D_A}{k_B} > \frac{1}{2} + \frac{1}{2} = 1.
\end{align*}
  Thus if we could construct a  full $(k_A+1, k_B, D_A, 
D_B)$-graph $H$ with no IT, by removing a vertex from each $A$-class
we would obtain a full $(k_A, k_B, D_A, 
D_B)$-graph with no IT, as required. Hence we may assume $k_A$ is as
large as possible such that Condition~(\ref{ratio-condition}) holds,
and therefore  $k_A\geq 2D_B$.
\end{proof}

For a $(k_A, k_B, D_A, D_B)$-graph $H$ with $a$ $A$-classes and $b$
$B$-classes, the $A$-{\it deficit} (respectively $B$-{\it deficit}) of
$H$ is $ak_A - |A|$ (respectively $bk_B-|B|$). Thus each deficit is
non-negative, and $H$ is full if and only if the $A$-deficit and
$B$-deficit of $H$ are both zero. We will achieve our aim using
Lemma~\ref{join-bipartite} over many steps. 

The three main phases toward constructing our $(k_A,k_B,D_A,D_B)$-graph
$H$ of Theorem~\ref{thm:sharp} are summarized as follows: 
\begin{enumerate}[I.]
	\item Construct a $(k_A, k_B, D_A, D_B)$-graph $H_0$ with no
          IT, $A$-deficit $0$ and $B$-deficit less than $k_B$, and a
          $B$-class of size $D_A$. 
	\item Using $H_0$, construct a $(k_A, k_B, D_A, D_B)$-graph
          $H_1$ with no IT, $A$-deficit $0$, and $B$-deficit at most
          $D_A$. 
	\item Using $H_1$, construct a full $(k_A, k_B, D_A, D_B)$-graph
          $H$ with no IT, as required.
\end{enumerate}

\paragraph{Phase~I.}
Let $K$ denote the complete bipartite graph $K_{D_B,
  D_A}$. We will construct the graph $H_0$ via the following procedure,
which iteratively adds copies of $K$ while applying
Lemma~\ref{join-bipartite}. In Steps~(2) and~(3.1), the current graph
$H_0$ 
will play the role of $H$ in Lemma~\ref{join-bipartite} while the
added graph $K$ will play the role of $J$. This applies in similar
fashion to Step~(3.2). The only restriction we implicitly impose
during these distribution steps is that we do not increase any
$A$-class size beyond $k_A$ or $B$-class size beyond $k_B$, so that
$H_0$ remains a $(k_A, k_B, D_A, D_B)$-graph at all times.
\begin{itemize}
	\item[(1)] Let $H_0 = K$ with one $A$-class (of size $D_B$)
          and one $B$-class (of size $D_A$). 
	\item[(2)] If the $A$-deficit of $H_0$ is at least $D_B$: Add
          to $H_0$ a new copy of $K$, with a new $B$-class of $H_0$ consisting of the
          $B$-side of $K$, and the $A$-side of $K$ distributed among
          the existing $A$-classes of $H_0$. Rename the new graph
          $H_0$ and go to (2). 
	\item[(3)] Else if the $A$-deficit of $H_0$ is positive and
          less than $D_B$: 
	\begin{itemize}
		\item[(3.1)] If the $B$-deficit is at least $D_A$: Add
                  to $H_0$ a new copy of $K$, with a new $A$-class of $H_0$
                  consisting of the $A$-side of $K$, and
                  the $B$-side of $K$ distributed among the existing
                  $B$-classes of $H_0$. Rename the new graph 
          $H_0$ and go to (2).
		\item[(3.2)] Else (the $B$-deficit is less than
                  $D_A$): Let $d_A$ be the $A$-deficit. Add to $H_0$ a
                  copy of $K' = K_{d_A,D_A}$, with a new $B$-class of $H_0$
                  consisting of the $B$-side of $K'$, and
                  the $A$-side of $K'$ distributed among the existing
                  $A$-classes of $H_0$. Rename the new graph
          $H_0$. Stop. Return $H_0$.
       	\end{itemize}
	\item[(4)] Else (the $A$-deficit is $0$): Stop. Return $H_0$.
\end{itemize}

\begin{lem}\label{H0props}
The above procedure terminates, and the output graph $H_0$ is a
$(k_A,k_B,D_A,D_B)$-graph with no IT, $A$-deficit $0$, and $B$-deficit
less than $k_B$. The last class added is a $B$-class, and its size is $D_A$. 
\end{lem}

\begin{proof}
Clearly the starting graph $H_0 = K$ is a $(k_A,k_B,D_A,D_B)$-graph
with no IT. By Lemma~\ref{join-bipartite}, $H_0$ has no IT throughout
the procedure. At each execution of Step~(2) (whenever the $A$-deficit is at
least $D_B$), we add a copy of $K$ to
lower the $A$-deficit of $H_0$ by $D_B$. Thus in particular, Step~(2)
is executed at most a finite number of times consecutively (actually
this number is at most $\lceil k_A/D_B\rceil$ but we do not need
this fact). At each
execution of Step~(3.1), we similarly add a copy 
of $K$ to lower the $B$-deficit of $H_0$ by $D_A$ if the $B$-deficit
is at least $D_A$. Thus, Steps~(2) and~(3.1) are always executed
in such a way that $H_0$ remains a $(k_A,k_B,D_A,D_B)$-graph. The same
is easily seen to be true after the execution of Step~(3.2), if it ever
occurs.

Next, assuming the above procedure terminates, we see that
whether it terminates at Step~(3.2) or at Step~(4), the $A$-deficit of
$H_0$ will be $0$. To show that the algorithm terminates and that the
$B$-deficit of $H_0$ will be less than $k_B$, we need the following
claim. Let $t$ be as in Observation~\ref{easyob}(1).

\begin{claim}
Consider any step of the procedure before Step~(3.2) is executed. Let
$a$ and $b$ denote the current numbers of $A$-classes and $B$-classes,
respectively, and let $d_A$ denote the current $A$-deficit of
$H_0$. Then the current $B$-deficit $d_B$ of $H_0$ is given by 
\begin{align*}
	d_B = k_B - \frac{d_A(k_B - D_A)}{D_B} - \frac{at}{D_B}.
\end{align*}
\end{claim}
\begin{claimproof}
By the definition of $A$-deficit, we have $|A| = ak_A - d_A$. Hence the
number of copies of $K = K_{D_B, D_A}$ is $(ak_A - d_A)/D_B$. On the
other hand, the number of copies of $K$ is exactly $a + b - 1$, since we
start with one copy and $a=b=1$
and then add one new copy plus one new $A$-class or $B$-class at each execution of Step~(2) or
(3.1). Thus $a + b - 1 = (ak_A - d_A)/D_B$ and the $B$-deficit is 
\begin{align*}
	d_B &= bk_B - (a + b - 1)D_A = (a + b - 1)(k_B - D_A) - (a - 1)k_B \\
	&= \frac{ak_A - d_A}{D_B} (k_B - D_A) - (a - 1)k_B = a \left(
        \frac{k_A(k_B - D_A)}{D_B} - k_B \right) - \frac{d_A(k_B -
          D_A)}{D_B} + k_B \\ 
	&= -\frac{a(k_AD_A + k_BD_B - k_Ak_B)}{D_B} - \frac{d_A(k_B -
          D_A)}{D_B} + k_B = -\frac{at}{D_B} - \frac{d_A(k_B -
          D_A)}{D_B} + k_B. 
\end{align*}
This completes the proof of the claim.
\end{claimproof}
From the claim, we see that while Step~(3.2) has not been executed, the
$B$-deficit of $H_0$ is always less than $k_B$. Thus if the algorithm
terminates at Step~(4) then the $B$-deficit is less than $k_B$ as
needed. If Step~(3.2) is
executed, then the $B$-deficit will still be less than $D_A + (k_B -
D_A) = k_B$, so again we achieve the $B$-deficit requirement.

To show termination, observe that if $a > D_B(k_B - D_A)/t$, then for any
$d_A \ge 0$ we have  
\begin{align*}
	k_B - \frac{d_A(k_B - D_A)}{D_B} - \frac{at}{D_B} < k_B -
        \frac{d_A(k_B - D_A)}{D_B} - (k_B - D_A) \le D_A. 
\end{align*}
Therefore, the claim implies that unless Step~(4) or Step~(3.2) is executed
(which terminates the algorithm), after adding sufficiently many
$A$-classes, the $B$-deficit $d_B$ of $H_0$ will always be less than
$D_A$ in subsequent steps. Then after possibly executing Step~(2) some
more times, $H_0$ will then have $A$-deficit less than $D_B$ and
$B$-deficit less than $D_A$, the conditions under which Step~(3.2) is
executed. Thus the procedure terminates. 

Finally, consider the last class added in the procedure before
termination. If it is an $A$-class then by
Observation~\ref{easyob}(3), the new $A$-deficit is at least
$k_A-D_B\geq D_B$. Hence the conditions for termination (Steps~(3.2) or~(4)) do not hold. Thus the last class added is a $B$-class, and it is
added in an execution of Step~(2) or Step~(3.2). Hence it contains exactly $D_A$
vertices, as required.
\end{proof}

\paragraph{Phase~II.}
We obtain the graph $H_1$ by stringing together
multiple copies of the graph $H_0$, the output from Phase~I, and distributing vertices again via
Lemma~\ref{join-bipartite}. We use the following procedure: 


\begin{itemize}
	\item[(1)] Let $H_1 = H_0$ with the same partition classes.
	\item[(2)] If $H_1$ has $B$-deficit greater than $D_A$: Add to
          $H_1$ a copy of $H_0$ together with all its partition
          classes, except for one 
          $B$-class of size $D_A$, which is distributed among
          existing $B$-classes in $H_1$. Rename the new graph $H_1$
          and go to (2). 
	\item[(3)] Else (the $B$-deficit is at most $D_A$): Stop. Return $H_1$.
\end{itemize}

\begin{lem}
The above procedure terminates, and the output graph $H_1$ is a
$(k_A,k_B,D_A,D_B)$-graph with no IT, $A$-deficit $0$, and $B$-deficit
at most $D_A$. 
\end{lem}

\begin{proof}
  Again by Lemma~\ref{join-bipartite} the output graph has no IT.
  
Let $B_0$ denote the $B$-side of $H_0$ and $b_0$ its number of
$B$-classes. Then the $B$-deficit of $H_0$ is $b_0k_B - |B_0|$,
which is less than $k_B$ by Lemma~\ref{H0props}. Hence $|B_0| - (b_0 -
1)k_B > 0$. The $A$-deficit of $H_1$ is $0$ at 
every step of the procedure since the $A$-deficit of each added copy
of $H_0$ is zero. On the other hand, at every execution of Step~(2) of
the procedure, we add 
$b_0 - 1$ new $B$-classes, and an additional $|B_0|$ vertices. Moreover,
since the condition for Step~(2) is that the current graph has
$B$-deficit greater than $D_A$, all $D_A$ vertices of the chosen
$B$-class in the new copy of
$H_0$ can be distributed among the $B$-classes of $H_1$ such that the
result is still a $(k_A,k_B,D_A,D_B)$-graph. Thus
after the $j$th iteration of Step~(2), we inductively obtain a
$(k_A,k_B,D_A,D_B)$-graph with
$B$-deficit
$$d_B^j=(j(b_0-1)+b_0)k_B - (j + 1)|B_0|=b_0k_B-|B_0|-j(|B_0| - (b_0 -
1)k_B).$$ 
Since  $|B_0| - (b_0 - 1)k_B > 0$, the $B$-deficit decreases at every
execution of Step~(2). Thus Step~(2) cannot be executed more than
$j=\ceilfrac{b_0k_B-|B_0|}{|B_0| - (b_0 - 1)k_B}$ times, and at
some iteration the $B$-deficit becomes at most $D_A$. Thus the
procedure terminates and outputs a graph with the required properties.
\end{proof}

\paragraph{Phase~III.}
Let $d_B$ be the $B$-deficit of $H_1$, the output of Phase~II, and so
$d_B \le D_A$. If $d_B = 0$, then we set $H = H_1$ and we are
done. Otherwise $d_B > 0$, and we construct a graph $H_1'$ by adding to
$H_1$ a copy of $K' = K_{D_B, d_B}$, with a new $A$-class $A_0$ consisting of the
$A$-side of $K'$, and the $B$-side of $K'$ distributed among the existing
$B$-classes in $H_1$, as per Lemma~\ref{join-bipartite}. Then $H_1'$ is a $(k_A,k_B,D_A,D_B)$-graph with no IT,
$B$-deficit zero and $A$-deficit $k_A-D_B$, and only one $A$-class
$A_0$ that 
is not full. Now we construct the graph $H$ by
first setting $H = 
H_1'$ and iteratively adding to it a copy of $H_1'$ along with all its vertex
classes except for the corresponding copy of $A_0$, whose vertices are
all added into the existing class $A_0$ in $H$ (as per Lemma~\ref{join-bipartite}), until this class has size at
least $k_A$. This iteratively increases the only non-full class
$A_0$ until its size reaches or exceeds $k_A$. We finish by
deleting $|A_0| - k_A$ vertices from $A_0$ to ensure $|A_0|=k_A$. This
gives us a $(k_A, k_B, D_A, 
D_B)$-graph $H$ with no IT, $A$-deficit $0$, and $B$-deficit $0$, thus
proving Theorem~\ref{thm:sharp}. \hfill$\blacksquare$

%
%
%
%

%
%
%
%
%
%

\section{Sharpness in a very asymmetric setting}\label{df}

In this section, we make use of the construction of Szab\'o and
Tardos~\cite{SzTa06} to prove sharpness of Proposition~\ref{prop:CaKa22}. 

\begin{prop}\label{prop:CaKa22sharp}
For any positive integer $D$,
there exists a bipartite-cover $H=H(A,B,G,L)$ of maximum degree $D$
with no IT, such that $|L(v)|=2D^2-1$ for all $v\in A_G$, $|L(w)|=2$ for all $w\in B_G$, and  no vertex of $A_H$ is adjacent to both vertices of $L(w)$ for some $w\in B_G$.
\end{prop}


\begin{proof}
Let $J$ be a bipartite graph with vertex classes $A$ and $B$, where
$B$ is partitioned into designated pairs $(x,y)$, and no vertex of $A$ is
adjacent to both $x$ and $y$ in any pair of $B$. Define the graph
$H'(J)$ from $J$ as follows. Set $V(H'(J))=A$. For each pair $(x,y)$ in
$B$, put an edge of $H'(J)$ joining each vertex of $N_J(x)$ to each
vertex of $N_J(y)$. We say each of these edges is {\it created} by the
pair $(x,y)$, so that $(x,y)$ creates a complete bipartite
subgraph of $H'(J)$ with vertex classes $N_J(x)$ and $N_J(y)$ (noting
that the condition on $J$ implies $N_J(x)\cap N_J(y)=\emptyset$). Note
that if $d_J(x)=d_J(y)=D$ then this complete bipartite subgraph of
$H'(J)$ is a $K_{D,D}$.
Given $D$, set $m=D^2$. The following claim was essentially proven in~\cite{CaKa22}.

\begin{claim}
There exists $J$ with $\Delta(J)=D$ such that $H'(J)=K_{m,m}$. 
\end{claim}

\begin{claimproof}
Choose the $A$-side of $J$ to be a set of size $2D^2$, partitioned into
$A_1\cup\ldots\cup A_D\cup C_1\cup\ldots\cup C_D$ where
$|A_i|=|C_j|=D$ for all $1\leq i,j\leq D$. Choose the $B$-side of $J$ to consist
of $D^2$ pairs $B_{ij}=(x_i,y_j)$. For each $(i,j)$ join $x_i\in B_{ij}$ to all of
$A_i$ and $y_j\in B_{ij}$ to all of $C_j$. Then $d_J(x_i)=d_J(y_j)=D$,
and each edge $e=zw$ of
$K_{m,m}$ with vertex classes $A'=A_1\cup\ldots\cup A_D$ and
$C'=C_1\cup\ldots\cup C_D$ is created by exactly one $B_{ij}$, namely
that for which $z\in A_i$ and $w\in C_j$.

The neighbourhood in $J$ of each $a\in A_i$ is the set of all $x_i\in
B_{ij}$ for $1\leq j\leq D$, so $d_J(a)=D$. Similarly for each $c\in
C_j$. Thus $\Delta(J)=D$.
\end{claimproof}

By taking $2m-1$ disjoint copies of the previous construction we get:

\begin{claim}
There exists $H$ with $\Delta(H)=D$ such that $H'(H)$ is the disjoint
union of $2m-1$ copies of $K_{m,m}$.
\end{claim}

Let $A$ and $B$ denote the vertex classes of $H$ from the last claim,
where $B$ is partitioned
into pairs as given. Now give a new partition of $A=V(H'(H))$ into
sets of size $2m-1$ ---let us call them {\it blocks}--- such that the graph
$H'(H)$ has no IT with respect to the blocks. This is possible since
the construction of Szab\'o and Tardos~\cite{SzTa06} (see Construction 3.3 therein) is itself a
disjoint union of $2m-1$ copies of $K_{m,m}$. We claim that $H$ has no
IT with respect to this partition.

Suppose on the contrary that $T$ is an IT of $H$. Let
  $T|_A$ denote the subset of $T$ that lies in blocks of $A$. Since
  $H'(H)$ has no IT, some edge $zw$ of $H'(H)$ joins two vertices in
  $T|_A$. By definition of $H'(H)$, this means that $\{z,w\}$ is
  matched to $\{x,y\}$ for some designated pair $(x,y)$ in $B$. But
  then $T$ cannot contain a vertex of $\{x,y\}$, contradicting that $T$ is an IT of $H$.
\end{proof}

\subsection*{Acknowledgements}

This work started in April of 2022 during a visit by PH to Nijmegen, where RJK was then affiliated.
We are grateful to the anonymous referees for their careful reading of the manuscript.

\paragraph{Open access statement.} For the purpose of open access,
a CC BY public copyright license is applied
to any Author Accepted Manuscript (AAM)
arising from this submission.

\bibliographystyle{abbrv}
\bibliography{bitrans}

\end{document}